\documentclass[11pt]{article}
\usepackage{mathtools,latexsym,color}
\usepackage{amsmath,amsthm,amssymb,amscd,amsfonts}
\usepackage{ae,aecompl}

\newtheorem{theo}{Theorem}

\newtheorem{cor}[theo]{Corollary}
\newtheorem{prop}[theo]{Proposition}

\def\R{\R}

\def\R{{\mathbb R}}

\def\qed{\hfill $\vcenter{\hrule height .3mm
\hbox {\vrule width .3mm height 2.1mm \kern 2mm \vrule width .3mm
height 2.1mm} \hrule height .3mm}$ \bigskip}

\def\lam{\lambda}

\def\to{\rightarrow}

\newcommand \supp{\operatorname{supp} \,}
\def\pmx{\begin{pmatrix}}
\def\emx{\end{pmatrix}}
\def\Hess{{\nabla^2}}

\def\det{{\rm det}}
\def\supp{{\hbox{supp}}}
\def\Supp{{\hbox{supp}}}

\def\R{\mathbb R}

\begin{document}

\title{Stability results  for some  geometric inequalities and their functional versions
\footnote{Keywords: entropy, divergence, affine isoperimetric inequalities, log Sobolev inequalities. 2010 Mathematics Subject Classification: 46B, 52A20,  60B, 35J}}

\author{Umut Caglar and Elisabeth M.  Werner\thanks { Partially supported by an  NSF grant  } }

\date{}

\maketitle
\begin{abstract}
The Blaschke Santal\'o inequality and the $L_p$ affine isoperimetric inequalities are major inequalities in convex geometry and they have  a wide range of applications.
Functional versions of the Blaschke Santal\'o  inequality  have been established over the years through many contributions. More recently and ongoing, such functional 
versions have been established for the $L_p$ affine isoperimetric inequalities as well. These functional versions involve  notions from information theory, like entropy and 
divergence. 
\par
We  list  stability versions for  the geometric inequalities as well as for their functional counterparts. Both are  known for the Blaschke Santal\'o inequality.
Stability versions for the $L_p$ affine isoperimetric inequalities  in the case of convex bodies have only been known in all dimensions for $p=1$ and for $p > 1$ only  for convex bodies in the plane. Here, we prove almost optimal   stability results for the $L_p$ affine isoperimetric inequalities, for all $p$, for all  convex bodies, for all dimensions. 
Moreover, we give stability versions for the corresponding functional versions of the $L_p$ affine isoperimetric inequalities, namely the  reverse log Sobolev inequality, 
the $L_p$ affine isoperimetric inequalities for  log concave functions and certain divergence inequalities.

\end{abstract}

\section{Introduction and Background}

We present stability results for several geometric and functional inequalities. 
Our  main focus will be
on geometric  inequalities  coming from affine convex geometry,  namely the Blaschke Santal\'o inequality, e.g., \cite{GardnerBook, SchneiderBook}, and the $L_p$  affine-isoperimetric  and related inequalities \cite{Bl1, Deicke, Lutwak1996, PaourisWerner2011, Werner-Ye} and also their functional counterparts, which includes the functional Blaschke Santal\'o inequality \cite{ArtKlarMil, KBallthesis,  Fradelizi+Meyer, Lehec2009} and the recently established divergence and entropy inequalities \cite{ArtKlarSchuWer, CaglarWerner2014, CFGLSW}.
These inequalities are fundamental in convex geometry  and geometric analysis, e.g.,  \cite{Bernig:2014, Haberl2012, HabSch2, Lutwak1996, LYZ2002, MW1, MW2, SW2004, Werner2007, Werner-Ye, WY2} and  they have applications throughout mathematics. We only quote: approximation theory of convex bodies by polytopes \cite{Boe, Gruber1993, Ludwig1999,  Reitzner, Schuett1991, SW2002}, affine curvature flows \cite{Andrews1, Andrews2, Stancu2002, Stancu2003}, information theory \cite{ArtKlarSchuWer,  CaglarWerner2014, CaglarWerner2015, CFGLSW, PaourisWerner2011, Werner2012/1}, valuation theory \cite{Alesker1999, Haberl2011, Haberl2012,  Ludwig2003, Ludwig2010, 
Ludwig2010/2, Ludwig-Reitzner, Parapatits2013, Schu}  and partial differential equations \cite{Lu-O}.
Therefore,  it is important to know stability results of those inequalities.
\par
Stability results answer the following question: Is the inequality that we consider sensitive to small perturbations? In other words, if a function almost attains the equality in a given inequality, is it possible to say  that then this function is close to the minimizers of the inequality?
For the Blaschke Santal\'o inequality and the functional Blaschke Santal\'o inequality such stability results have been established in \cite{BallBo} and \cite{BaBoFr},  respectively.
Stability results for the $L_p$-affine isoperimetric inequalities for convex bodies were proved  in \cite{Boeroeczky} for $p=1$  and dimension $n \geq 3$. In
 \cite{Ivaki1, Ivaki2}, stablility results  for the $L_p$-affine isoperimetric inequality were proved  in dimension $2$  and for  $p \geq 1$.
 \par
We  present here  stability results for the $L_p$-affine isoperimetric inequalities for all $p$ and in all dimensions. 
Stability results for the corresponding functional versions of these inequalities are also given.
\vskip 2mm
Throughout, we will assume that $K$ is a convex body in $\mathbb{R}^n$, i.e., a convex compact subset of $\mathbb{R}^n$ with non-empty interior $\text{int}(K)$. We denote by $\partial K$ the boundary of $K$ and by $\text{vol}(K)$ or $|K|$ its 
$n$-dimensional volume. $B^n_2$ is  the Euclidean unit ball  centered at $0$ and $S^{n-1}= \partial B^n_2$ its boundary. The standard inner product on $\mathbb{R}^n$ is $\langle , \rangle$. 
It induces the  Euclidean norm, denoted by $\|\cdot\|_2$.
We will  use the Banach-Mazur distance  $d_{BM}(K,L) $ to measure the distance between  the convex bodies $K$ and $L$,  
\begin{eqnarray*}
d_{BM}(K,L)= 
\min\{\alpha \geq 1:  K-x  \subset T (L-y) \subset   \alpha (K-x),\   \text{for}   \  T \in GL(n),  x, y \in \mathbb{R}^n\}.
\end{eqnarray*}
In the case when $K$ and $L$  are $0$-symmetric, $x$ and $y$ can be taken to be $0$, 
\begin{eqnarray*}
d_{BM} (K,L)= 
\min\{\alpha \geq 1:  K  \subset T (L) \subset   \alpha \  K,\   \text{for}   \  T \in GL(n) \}.
\end{eqnarray*}

\section{Stability in inequalities for  convex bodies}

\subsection {The Blaschke Santal\'o inequality}
\par
Let $K$ be a  convex body in $\mathbb{R}^n$ such that $0 \in \text{int}(K)$. The polar $K^\circ$ of $K$ is defined as
$$
K^{\circ}=\{y\in {\R}^n\;:\;\langle x,y\rangle \le 1 \; \mbox{for all }\; x\in K\}
$$
and, more generally, the polar $K^z$  with respect to $z \in \text{int}(K)$ by $(K-z)^\circ$. The classical Blaschke Santal\'o inequality (see, e.g., \cite{SchneiderBook}) states that
there is a unique point $s \in \text{int}(K)$, the Santal\'o point of $K$,  such that the volume product $|K| |K^s|$ is minimal and that 
$$|K| \,|K^s| \leq |B^n_2|^2$$ with equality if and only if $K$ is an  ellipsoid.
\par
Ball and B\"or\"oczky  \cite{BallBo} proved the following stability version of the Blaschke Santal\'o inequality. It will be one of the  tools to prove stability versions for the $L_p$-affine isoperimetric inequalities.
\par
\begin{theo}\cite{BallBo} \label{stability:BS}
Let $K$ be a convex body in $\mathbb{R}^n$, $n \geq  3$, with Santal\'o point  at $0$. If 
$|K| |K^\circ| > (1- \varepsilon ) |B^n_2|^{2}$, for $\varepsilon \in (0, \frac{1}{2})$, 
then for some $\gamma > 0$,  depending only on $n$, we have
$$
d_{BM}(K,B^n_2) < 1+ \gamma \varepsilon^{\frac{1}{3(n+1)}} | \log \varepsilon | ^\frac{4}{3(n+1)}.
$$
\end{theo}
\noindent
{\bf Remark.} It was noted in \cite{BallBo} that if $K$ is  $0$-symmetric, then the exponent $\frac{1}{
3(n+1)}$ occurring in Theorem \ref{stability:BS}
 can be replaced by $\frac{2}{3(n+1)}$. Moreover, it was also noted in \cite{BallBo} that 
taking $K$ to be the convex body resulting from $B^n_2$ by cutting off two
opposite caps of volume $\varepsilon$,  shows that the exponent $\frac{1}{(3(n + 1)}$ cannot be
replaced by anything larger than $\frac{2}{n + 1}$,  even for $0$-symmetric convex
bodies with axial rotational symmetry. Therefore the exponent of $\varepsilon$  is of the
correct order. 

\subsection {$L_p$-affine isoperimetric inequalities}

Now we turn to stability results for the $L_p$-affine isoperimetric inequalities for convex bodies. These inequalities involve the $L_p$-affine surface areas
which are a central part of the rapidly developing $L_p$ and Orlicz  Brunn Minkowski theory  and are the focus of intensive investigations (see, e.g.,  \cite{CaglarYe2015}, 
\cite{GaZ}, \cite{Ga3}, \cite{GardnerHugWeil}, 
\cite{HabSch2},  
\cite{Ludwig2006}-\cite{LutwakYangZhang2004},  
\cite{Schu}-\cite{SW2002}, \cite{Stancu2002}, \cite{Stancu2003}, 
\cite{Werner2007}-\cite{Werner-Ye}). 
\par
The $L_p$-affine surface area $as_p(K)$ of a convex body
$K$ in $\mathbb{R}^n$ was  introduced by Lutwak for all $p>1$ in his seminal paper \cite{Lutwak1996} and for all other $p$ by Sch\"utt and Werner
\cite{SW2004}(see also \cite{Hug1996}). The case $p=1$ is the classical affine surface area introduced by Blaschke in dimensions $2$ and $3$  \cite{Bl1} (see also \cite{Leicht, SW1990}). 
\par
Let $p \in \mathbb{R}$, $p\neq -n$
and assume that $K$ is a convex body with  centroid or  Santal\'o point  at the origin.  Then
\begin{equation}\label{asp}
as_p(K)= \int_{\partial K} \frac{\kappa(x)^{\frac{p}{n+p}}}{\langle x, N(x)\rangle^{\frac{n(p-1)}{n+p}}}d\mu_K(x), 
\end{equation}
where $N(x)$ is the unit outer normal in  $x \in \partial K$, the boundary of $K$, $\kappa(x)$ is the (generalized) Gaussian curvature in $x$ and $\mu_K$ is the surface area measure on $\partial K$. In particular, for $p=0$
$$
as_{0}(K)=\int_{\partial K} \langle x,N_{ K}(x)\rangle
\,d\mu_{K}(x) = n|K|.
$$
For $p=1$, 
$$
as_{1}(K)=\int_{\partial K}\kappa_K(x)^{\frac{1}{n+1}} d\mu_{ K}(x)
$$
is the classical affine surface area  which is independent
of the position of $K$ in space. Note also  that $as_p(B^n_2) = \text{vol}_{n-1} (\partial B^n_2) = n |B^n_2|$ for all $p \neq -n$. 
If the boundary of $K$ is sufficiently smooth, (\ref{asp})
can be written as an integral over the boundary
$S^{n-1}$ of the Euclidean unit ball $B^n_2$,
$$
as_{p}(K)=\int_{S^{n-1}}\frac{f_{K}(u)^{\frac{n}{n+p}}}
{h_K(u)^{\frac{n(p-1)}{n+p}}}
d\sigma(u).
$$
Here, $\sigma$ is the usual surface area measure on $S^{n-1}$, 
$h_{K}(u)= \max_{x \in K} \langle x, u \rangle $ is the support function of $K$ in direction $u\in S^{n-1}$,
and $f_{K}(u)$ is the curvature function, i.e. the reciprocal of
the Gaussian curvature $\kappa _K(x)$ at this point $x \in
\partial K$ that has $u$ as outer normal. In particular, for
$p=\pm \infty$,
\begin{equation}\label{inf-aff}
as_{\pm\infty}(K)
=\int_{S^{n-1}}\frac{1}{h_K(u)^{n}}
d\sigma(u)
=n|K^{\circ}|.
\end{equation}
\par
The $L_p$-affine surface area is invariant under linear transformations $T$ with determinant $1$. More precisely, (see, e.g., \cite{SW2004}), if $T:\mathbb{R}^n \rightarrow \mathbb{R}^n$ is a linear,  invertible map, then
\begin{equation}\label{invariant}
as_p(T(K)) =| \det T| ^\frac{n-p}{n+p} as_p(K).
\end{equation}
The $L_p$-affine surface area is a valuation \cite{Ludwig2010, Ludwig-Reitzner, S1993},  i.e., for convex bodies $K$ and $L$ such that $K \cup L$ is convex,
$$
as_p(K  \cup L) + as_p(K  \cap L)= as_p(K) + as_p(L).
$$
Valuations have become a major topic in convex geometry in recent years. We refer to e.g.,  \cite{ Alesker1999, Haberl2011, Haberl2012,  Ludwig2003, Ludwig2010, 
Ludwig2010/2, Ludwig-Reitzner, Parapatits2013, Schu}.
\par
We now state the  $L_p$-affine isoperimetric inequalities for the quantities $as_p(K)$. They  were  proved by Lutwak for $p>1$ \cite{Lutwak1996} and for all other $p$
by Werner and Ye \cite{Werner-Ye}. The case $p=1$ is the classical affine isoperimetric inequality \cite{Bl1, Deicke}.
\begin{theo} ($p=1$ \cite{Bl1, Deicke}, $p>1$ \cite{Lutwak1996}, all other $p$ \cite{Werner-Ye}) \label{p-aff-iso4}
Let $K$ be a convex
body with centroid at the origin.
\par
(i)
If $p > 0$,  then
\begin{eqnarray*}
\frac{as_p(K)}{as_p(B^n_2)}\leq \left(\frac{|K|}{| B^n_2|
}\right)^{\frac{n-p}{n+p}},
\end{eqnarray*}
with equality if and only if $K$ is an ellipsoid. For $p=0$, equality holds trivially for all $K$.
\par
(ii) If $-n<p<0$, then
\begin{eqnarray*}
\frac{as_p(K)}{as_p(B^n_2)}\geq
\left(\frac{|K|}{|B^n_2|}\right)^{\frac{n-p}{n+p}},
\end{eqnarray*} with
equality if and only if $K$ is an ellipsoid.
\par
(iii) If  $K$ is in addition  in $C^2_+$ and if $p < -n$, then
\begin{equation*}
c^{\frac{np}{n+p}}\left(\frac{|K|}{|B^n_2|}\right)^{\frac{n-p}{n+p}} \leq \frac{as_p(K)}{as_p(B^n_2 )}.
\end{equation*}
\end{theo}
The constant $c$ in (iii) is the constant from the inverse
Blaschke Santal\'o inequality due to Bourgain and Milman \cite{BM}.  This constant has recently been improved by Kuperberg \cite{GK2}  (see also \cite{ Nazarov} for a different proof).

\subsection{Stability for the $L_p$-affine isoperimetric inequality for convex bodies}

Stability results for the $L_p$-affine isoperimetric inequalities for convex bodies were proved by B\"or\"oczky \cite{Boeroeczky} for $p=1$  and dimension $n \geq 3$. 
Ivaki \cite{Ivaki1, Ivaki2} gave stability results for the $L_p$-affine isoperimetric inequality in dimension $2$  and  $p \geq 1$.
We  present here  stability results for the $L_p$-affine isoperimetric inequalities for all $p$ and in all dimensions. 
Before we do so, we first quote the results by B\"or\"oczky \cite{Boeroeczky} and Ivaki \cite{Ivaki2}.

\begin{theo} \cite{Boeroeczky}
If $K$ is a convex body in $\R^n$,  $n \geq 3$, and 
\begin{equation}
\left( \frac{as_1 (K)}{as_1 (B^n_2)} \right)^{n+1} > (1 - \epsilon )  \left( \frac{| K |}{ | B^n_2|}\right)^{n-1}     \ \ \ \text{for} \ \epsilon \in \left(0, \frac{1}{2}\right),
\end{equation}
then for some $\gamma > 0$,  depending only on $n$, we have
$$
d_{BM}(K,B^n_2) < 1+ \gamma \varepsilon^{\frac{1}{6n}} | \log \varepsilon | ^\frac{1}{6}.
$$
\end{theo}
Later, in \cite{BallBo}, the above approximation was improved to 
$$
d_{BM}(K,B^n_2) < 1+ \gamma \varepsilon^{\frac{1}{3 (n+1)}} | \log \varepsilon | ^\frac{4}{3(n+1)}.
$$
Ivaki \cite{Ivaki2} gave  a stability version for the Blaschke Santal\'o inequality from which the following stability result for the $L_p$-affine isoperimetric inequality in dimension $2$  and  $p \geq 1$ follows easily.
\begin{theo} \cite{Ivaki2}
Let $K$ be an origin symmetric  convex body in $\R^2$, and  $p \geq 1$. There exists an $\epsilon_p> 0$, depending on $p$, such that the
following holds.   If for an $\epsilon$, $ 0 < \epsilon < \epsilon_p$,
\begin{equation*}
\left( \frac{as_p (K)}{2 \pi} \right)^{p+2} > (1 - \epsilon)^{p} \left( \frac{\text{area}(K)}{ \pi}\right)^{2-p}     
\end{equation*}
then for some $\gamma > 0$,  we have
\begin{equation}\label{dim2}
d_{BM}(K,B^2_2) < 1+ \gamma \varepsilon^{\frac{1}{2}} .
\end{equation}

\end{theo}
The same author also considered the case  when $K$ is  a not necessarily origin symmetric  convex body in $\R^2$ \cite{Ivaki2}.  Then the order of approximation becomes $\frac{1}{4}$ instead of $\frac{1}{2}$.
Note also  that there are results in  dimension $n=2$  by B\"or\"oczky and Makai \cite{BorMak} on stability of the Blaschke Santal\'o inequality, 
from which a stability result of the form (\ref{dim2}) for the $L_p$-affine isoperimetric inequality
in dimension $2$ follows easily. But the order of approximation in the origin-symmetric case is $1/3$ and
in the general case  $1/6$.
\par
We now present  almost optimal  stability results for the $L_p$-affine isoperimetric inequalities, for all $p$, for all  convex bodies, for all dimensions. 
To do so, we use the above  stability  version of the Blaschke Santal\'o inequality  by 
Ball and B\"or\"oczky  \cite{BallBo}, together with  inequalities proved in \cite{Werner-Ye}.
\vskip 3mm
\begin{theo} \label{stab-affinebodies2}
Let $K$ be  a convex body in $\mathbb{R}^n$, $n \geq  3$, with Santal\'o point  or centroid at $0$.
\par
\noindent
(i) Let  $p > 0$. If  $\left(\frac{as_p(K)}{as_p(B^n_2)}\right) ^{n+p} > (1 - \varepsilon)^p \left(\frac{|K|}{|B^n_2|} \right)^{n-p}$, 
then for some $\gamma > 0$,  depending only on $n$, we have
$$
d_{BM}(K,B^n_2) < 1+ \gamma \varepsilon^{\frac{1}{3(n+1)}} | \log \varepsilon | ^\frac{4}{3(n+1)}.
$$
\par
\noindent
(ii) Let $-n < p <0$.  If  $\left(\frac{as_p(K)}{as_p(B^n_2)}\right)^{n+p} < (1 - \varepsilon)^p \left(\frac{|K|}{|B^n_2|} \right)^{n-p}$, 
then for some $\gamma > 0$,  depending only on $n$, we have
$$
d_{BM}(K,B^n_2) < 1+ \gamma \varepsilon^{\frac{1}{3(n+1)}} | \log \varepsilon | ^\frac{4}{3(n+1)}.
$$
\end{theo}
\vskip 2mm
\noindent
{\bf Remarks.}
(i)  If $K$ is $0$-symmetric, then $\varepsilon^{\frac{1}{3(n+1)}}$ can be replaced by  $\varepsilon^{\frac{2}{3(n+1)}}$. This follows from \cite{BallBo}.
See also the Remark after Theorem \ref{stability:BS}.
\par
(ii)  The example in \cite{BallBo} already quoted in the Remark after Theorem \ref{stability:BS}  shows that $\varepsilon^{\frac{1}{3(n+1)}}$ cannot be replaced by anything smaller than $\varepsilon^{\frac{2}{n-1}}$, even for $0$-symmetric convex bodies with axial rotational symmetry. Indeed, let $K$ be the convex body obtained from $B^n_2$ by removing two opposite caps of volume $\varepsilon$ each. Then
$$\left(\frac{as_p(K)}{as_p(B^n_2)}\right) ^{n+p} > (1 - k \varepsilon^\frac{n-1}{n+1})^p \left(\frac{|K|}{|B^n_2|} \right)^{n-p} = (1 -  \delta)^p \left(\frac{|K|}{|B^n_2|} \right)^{n-p},$$
where we have put $\delta = k \varepsilon^\frac{n-1}{n+1}$ and where $k$ is a constant that depends on $n$ only,  except for $0 <p<n$, where it also depends on $p$.
And $d_{BM}(K,B^n_2)= 1 + \gamma \delta^\frac{2}{n-1}$.
\vskip 3mm
\begin{proof}[Proof of Theorem \ref{stab-affinebodies2}]
(i) As $as_p(B^n_2)= n |B^n_2|$,   we observe that the inequality $$\left(\frac{as_p(K)}{as_p(B^n_2)}\right)^{n+p} > (1 - \varepsilon)^p \left(\frac{|K|}{|B^n_2|} \right)^{n-p}$$ is equivalent to
the inequality
\begin{equation}\label{assumption:p}
as_p(K) ^{n+p} > (1 - \varepsilon)^p n^{n+p}  |K| ^{n-p}  |B^n_2| ^{2p}.
\end{equation}
It was proved in \cite{Werner-Ye} that for all $p>0$, $$as_p(K) ^{n+p} \leq   n^{n+p} |K| ^{n}  |K^\circ| ^{p}.$$
Hence we get from the assumption that 
\begin{eqnarray*}
n^{n+p} |K| ^{n}  |K^\circ| ^{p} >  (1 - \varepsilon)^p n^{n+p}  |K| ^{n-p}  |B^n_2| ^{2p},
\end{eqnarray*}
or equivalently, that
\begin{eqnarray*}
 |K|   |K^\circ|  >  (1 - \varepsilon) \   |B^n_2| ^{2},
\end{eqnarray*}
and we conclude with the Ball and B\"or\"oczky  stability result in Theorem \ref{stability:BS}.
\par
(ii) The  proof of (ii) is done similarly. We use the inequality  $$as_p(K) ^{n+p} \geq   n^{n+p} |K| ^{n}  |K^\circ| ^{p},$$ which holds for $-n <p <0$ and which was also proved in \cite{Werner-Ye}.
\end{proof}
\vskip 3mm
Another stability result for the $L_p$-affine isoperimetric inequalities for convex bodies is obtained as a corollary to Proposition \ref{stab-affinefunctions} below.
We list it now, as we want to compare the two. Let  $K$ be a convex body in $\mathbb{R}^n$ with $0$ in its interior and 
let the function $\psi$ of Proposition \ref{stab-affinefunctions} be $\psi(x) = \|x\|_K^2/2$, where $\| \cdot\|_K$ is the gauge function of the  convex body $K$, 
\begin{equation*}\label{gauge}
\|x\|_K = \min\{ \alpha\geq 0: \ x \in  \alpha K\} = \max_{y \in K^\circ} \langle x, y \rangle.
\end{equation*}
Let
 \begin{equation}\label{affine0}
 as_\lambda(\psi) = \int_{\mathbb{R}^n}   e^{(2\lambda - 1)\psi(x) - \lambda  \langle \nabla \psi, x\rangle } \left(  \det \left(  \nabla^2 \psi(x) \right)\right)
^\lambda dx
\end{equation}
be the $L_\lambda$-affine surface area of the function $\psi$.  This quantity is discussed in detail in Section 3.3.
Differentiating  $\psi(x) = \|x\|_K^2/2$,  we get  $\langle x, \nabla \psi(x) \rangle = 2 \psi(x)$.
Thus,  for  $\psi(x) = \|x\|_K^2/2$, the expression (\ref{affine0})  simplifies to 
\begin{equation}\label{2homoasa0}
as_\lam(\psi)=\int_{\mathbb{R}^n} \left(\det \, \Hess  \psi (x)\right)^\lam e^{-\psi(x)}dx.
\end{equation}
Note that for the Euclidean norm $\|. \|_2$, 
$
as_\lam\left(\frac{\|\cdot\|_2^2}{2}\right) = \left(2 \pi\right)^\frac{n}{2}
$ and  it was proved in \cite{CFGLSW} that 
\begin{equation}\label{lambda=p}
\frac{as_\lam\left(\frac{\|\cdot\|_K^2}{2}\right)}{as_\lam\left(\frac{\|\cdot\|_2^2}{2}\right)} = \frac{as_p(K)}{as_p(B^n_2)}, 
\end{equation}
where $\lam$ and $p$ are related by $ \lam=\frac{p}{n+p}$.
Together with Proposition \ref{stab-affinefunctions}, this immediately implies another stability result for the $L_p$-affine isoperimetric inequalities for convex bodies.
\vskip 2mm
\begin{cor}\label{stab-affinebodies1}
Let $K$ be a convex body in $\mathbb{R}^n$ with the centroid  or  the Santal\'o point at the origin.
\par
\noindent
(i) Let $0 < p \leq \infty$ and  and suppose that for some $\varepsilon \in (0, \varepsilon_0)$, 
$$
\frac{as_p(K)}{as_p(B^n_2)} > (1-\varepsilon)^\frac{p}{n+p} \left(\frac{ |K|}{|B^n_2|}\right) ^\frac{n-p}{n+p}.
$$ 
\par
\noindent
(i) Let $-n <  p < 0$ and  and suppose that for some $\varepsilon \in (0, \varepsilon_0)$, 
$$
\frac{as_p(K)}{as_p(B^n_2)} < (1-\varepsilon)^\frac{p}{n+p} \left(\frac{ |K|}{|B^n_2|}\right) ^\frac{n-p}{n+p}.
$$ 
Then, in both cases (i) and (ii),  there exists $c>0$ and a positive definite matrix $A$ such that 
$$
\int_{R(\varepsilon) B^n_2} \left|  \|Ax\|^2_K - \|x\|_2^2 -c \right| dx < \eta \varepsilon ^{\frac{1}{129 n^2}}, 
$$
where $ R(\varepsilon) =\frac{|\log \varepsilon|^\frac{1}{2}}{8n} $ and $\varepsilon_0, \eta$ depend on $n$.
\end{cor}
\vskip 2mm 
\begin{proof}  It is easy to see (e.g., \cite{CFGLSW}) that $$|K| = \frac{1}{2^\frac{n}{2} \Gamma\left(1+\frac{n}2\right)} \int e^{-\frac{\|x\|^2_K}{2}} dx .$$ 
As $|B_2^n|=\frac{\pi^{\frac{n}{2}}}{\Gamma\left(1+\frac{n}2\right)}$, we get, with $\psi(x) = \frac{\|x\|^2_K}{2}$,  by (\ref{lambda=p}) and the assumptions of the theorem, that  for $0 < p \leq \infty$, 
$$as_\lambda (\psi) > (1-\varepsilon) ^\lambda \left( 2 \pi \right)^{n \lambda} \left(  \int e^{-\psi(x)} dx \right)^{1- 2 \lambda}.$$
We  have also used that  $\lam=\frac{p}{n+p}$.
The result for $0 < p \leq \infty$ then  follows  immediately from Proposition  \ref{stab-affinefunctions}. The case $-n <  p < 0$ is treated similarly.
\end{proof}

\vskip 3mm

\noindent
{\bf Remarks.}
In general, one cannot deduce Theorem \ref{stab-affinebodies2} from Corollary \ref{stab-affinebodies1}.  However, it follows from Theorem \ref{stab-affinebodies2} that there exists 
$T \in GL(n)$ and   $x_0,  y_0 \in \mathbb{R}^n$ such that 
$$K- x_0 \subset T(B^n_2 - y_0) \subset \left(1+ \gamma \varepsilon^{\frac{1}{3(n+1)}} | \log \varepsilon | ^\frac{4}{3(n+1)} \right) (K-x_0).
$$
For simplicity, assume that $x_0= y_0=0$, which corresponds to the case that $K$ is $0$-symmetric. Then this means that for all $x \in \mathbb{R}^n$,
$$
\left| \ \|x\|_K - \|T(x)\|_2 \  \right|  \leq \|T\|  \left(\gamma \varepsilon^{\frac{1}{3(n+1)}} | \log \varepsilon | ^\frac{4}{3(n+1)} \right) \|x\|_2
$$ 
and thus 
\begin{eqnarray*}
&&\hskip -10mm \int_{R(\varepsilon) B^n_2} \left| \ \|x\|^2_K - \|T(x)\|^2_2\ \right|  dx  \\
&\leq& \left(1+\gamma \varepsilon^{\frac{1}{3(n+1)}} | \log \varepsilon | ^\frac{4}{3(n+1)} \right)\   | B^n_2|  \  \|T\|^2  R^{n+2}(\varepsilon)\   \left(\gamma \varepsilon^{\frac{1}{3(n+1)}} | \log \varepsilon | ^{\frac{4}{3(n+1)} }\right)\\
&=&  \left(1+\gamma \varepsilon^{\frac{1}{3(n+1)}} | \log \varepsilon | ^\frac{4}{3(n+1)} \right)\  \frac{ | B^n_2|}{(8n)^{n+2}}  \  \|T\|^2  \left(\gamma \varepsilon^{\frac{1}{3(n+1)}} | \log \varepsilon | ^{\frac{4}{3(n+1)} +\frac{n+2}{2}}\right).
\end{eqnarray*}
Hence, allowing general $T$, the exponent  of  $\varepsilon$ can be improved.

\subsection{Stability result for the entropy power $\Omega_K$}

An  affine invariant quantity that is closely related to the  $L_p$-affine surface areas is the entropy power $\Omega_K$. It  was introduced in \cite{PaourisWerner2011} as the limit of $L_p$-affine surface areas, 
\begin{equation}\label{omega}
\Omega_K = \lim_{p \rightarrow \infty}\left( \frac{as_p(K)}{n |K^\circ|}\right)^{n+p}.
\end{equation}
The quantity $\Omega_K$ is related to the relative entropy of the cone measures of $K$ and $K^\circ$. We refer to \cite{PaourisWerner2011} for the details and 
only mention an affine isoperimetric inequality  for $\Omega_K$ proved in  \cite{PaourisWerner2011}.
\begin{theo} \label{omega} \cite{PaourisWerner2011}
 If $K$ is a convex body of volume $1$, then 
\begin{equation}\label{iso-omega}
\Omega_{K^\circ} \leq \Omega_{\left(\frac{B^n_2}{|B^n_2|^\frac{1}{n}}\right)^\circ}.
\end{equation}
Equality holds  if and only if $K$ is a normalized ellipsoid.
\end{theo}
\vskip 2mm
We now use the previous  theorems to prove  stability results for  inequality (\ref{iso-omega}). Using the invariant property (\ref{invariant}) and the fact that 
$as_p(B^n_2)= n |B^n_2|$,  this inequality  can be written as $$\Omega_{K^\circ} \leq |B^n_2|^{2n}.$$
\vskip 2mm
\begin{theo} \label{stability:Omega1}
Let $K$ be  a convex body in $\mathbb{R}^n$, $n \geq  3$, of volume $1$ and such that the  Santal\'o point  or the centroid  are at $0$. Suppose that for some $\varepsilon \in (0, \frac{1}{2})$, 
\begin{equation}\Omega_{K^\circ} > (1- \varepsilon)  |B^n_2|^{2n}.\label{kzero}\end{equation} Then for some $\gamma > 0$,  depending only on $n$, we have
$$
d_{BM}(K^\circ,B^n_2) < 1+ \gamma \left(\frac{2\varepsilon}{n}\right) ^{\frac{1}{3(n+1)}} \left| \log \frac{2\varepsilon}{n} \right| ^\frac{4}{3(n+1)}.
$$
\end{theo}
\vskip 2mm
Remarks similar to the ones  after Theorem \ref{stab-affinebodies2} hold.
\vskip 2mm
\begin{proof}
It was shown in \cite {Werner-Ye} that $\left(\frac{as_p(K^\circ)}{n |K|}\right)^{n+p}$ is decreasing in $p \in (0, \infty)$. By definition (\ref{omega}), 
$\lim_{p \rightarrow \infty} \left(\frac{as_p(K^\circ)}{n |K|}\right)^{n+p}  = \Omega_{K^\circ}$. Therefore
we  get with  assumption \eqref{kzero}  that for all $p >0$ 
$$
\left(\frac{as_p(K^\circ)}{n |K|}\right)^{n+p}  > (1 - \varepsilon)  |B^n_2| ^{2n}.
$$
Or, equivalently, as $|K|=1$, 
\begin{eqnarray*}
as_p(K^\circ)^{n+p}  &>& (1 - \varepsilon) n^{n+p} |K|^{n+p}   |B^n_2| ^{2n} = (1 - \varepsilon) n^{n+p}  |B^n_2| ^{2p} 
\,|B^n_2| ^{2(n-p)} \\
&\geq& (1 - \varepsilon) n^{n+p} |K^\circ|^{n-p}   |B^n_2| ^{2p}.
\end{eqnarray*}
In the last inequality we have used the Blaschke Santal\'o inequality $|K|\,|K^\circ|\le |B^n_2| ^{2}$,  which we can apply as long as $n-p \geq 0$.
Note that for all $\varepsilon\in(0,\frac{1}{2})$ and $p>0$ 
$$1 - \varepsilon> \left(1 -   \frac{2\varepsilon}{p}\right)^p.$$
Hence, using the elementary inequality above,  we get for all $0 <p \leq n$ that
\begin{eqnarray*}
as_p(K^\circ)^{n+p}   
>\left(1 -   \frac{2\varepsilon}{p}\right)^p n^{n+p} |K^\circ|^{n-p}   |B^n_2| ^{2p}.
\end{eqnarray*}
Inequality   (\ref{assumption:p}) and  the arguments used  after it,  imply that for all $0 <p \leq n$,
$$
d_{BM}(K^\circ,B^n_2) < 1+ \gamma \left(\frac{2 \varepsilon}{p}\right) ^{\frac{1}{3(n+1)}} \left| \log \frac{2 \varepsilon}{p}\right| ^\frac{4}{3(n+1)}.
$$
Since the right hand side of above equation is decreasing in $p,$ minimizing  over $p$ in the interval $(0,n]$ gives the result.
\end{proof}
\vskip 3mm
The second stability result and  the corresponding comparisons  (see the Remark after Corollary  \ref{stab-affinebodies1})   are  obtained accordingly. We skip the proof.
\vskip 2mm
\begin{theo} \label{stability:Omega2}
Let $K$ be  a convex body in $\mathbb{R}^n$, $n \geq  3$, of volume $1$ and with Santal\'o point  or centroid at $0$, such that 
$\Omega_{K^\circ} > (1- \varepsilon)  |B^n_2|^{2n}$. 
Then there exists $c>0$ and a positive definite matrix $A$ such that 
$$
\int_{R(\varepsilon) B^n_2} \left|  \|Ax\|^2_K - |x|_2^2 -c \right| dx < \eta \varepsilon ^{\frac{1}{129 n^2}}, 
$$
$ R(\varepsilon) =\frac{|\log \varepsilon|^\frac{1}{2}}{8n} $ and $\varepsilon_0, \eta$ depend on $n$.
\end{theo}

\section{Stability results for functional inequalities}

\subsection {Stability for the functional Blaschke Santal\'o inequality}

We will first state a functional version of the Blaschke Santal\'o inequality. To do so, we
recall that  the Legendre transform of a function
$\psi : \mathbb{R}^n \rightarrow \mathbb{R} \cup \{+\infty\} $  at $z  \in \mathbb{R}^n$  is defined by
\begin{equation}\label{legendre}
\mathcal{L}_z\psi(y) = \sup_{x \in \mathbb{R}^n} \left(\langle x-z,y \rangle   - \psi(x) \right), \ \  \text{for }  y \in \mathbb{R}^n.
\end{equation}
The function $L_z \psi : \mathbb{R}^n \rightarrow \mathbb{R} \cup \{+ \infty\}$  is always convex and lower
semicontinuous. If $\psi$ is convex, lower semicontinuous and $\psi < + \infty$, then
$L_z L_z \psi  = \psi $. 
When $z=0$, we write 
\begin{equation}\label{legendre0}
\psi^* ( y ) = \mathcal{L}_0\psi(y) = \sup_x \left( \langle{x,y\rangle} - \psi (x) \right).
\end{equation} 
\par
Work by K.M. Ball \cite{KBallthesis}, S. Artstein-Avidan, B. Klartag, V.D.Milman \cite{ArtKlarMil}, M. Fradelizi, M. Meyer \cite{Fradelizi+Meyer} and J. Lehec  \cite{Lehec2009} led to the functional
version of the Blaschke Santal\'o inequality which we now state.
\begin{theo} \cite{ArtKlarMil, KBallthesis,  Fradelizi+Meyer, Lehec2009}
Let $\rho:  \mathbb{R} \rightarrow \mathbb{R}_+$ be a log-concave non-increasing
function and $\psi: \mathbb{R}^n \rightarrow \mathbb{R} \cup \{+\infty\} $ be measurable.  Then
$$
\inf_{z \in \mathbb{R}^n}\int_{\mathbb{R}^n} \rho(\psi(x)) dx  \   \int_{\mathbb{R}^n} \rho(\mathcal{L}_z\psi(x)) dx \leq  \left( \int_{\mathbb{R}^n}  \rho\left(\frac{\|x\|_2^2}{2}\right) dx\right)^2.
$$
If $\rho$  is decreasing,  there is equality if and only if there exist $a$, $b$, $c$  in $\mathbb{R}$, $a < 0$, 
$z  \in \mathbb{R}^n$  and a positive definite matrix $A : \mathbb{R}^n \rightarrow \mathbb{R}^n$ such that
$$
\psi(x)  = \frac{ \|A(x + z)\|_2^2}{2}  +c, \ \  \text{for} \  \   x \in \mathbb{R}^n 
$$
and moreover either $c = 0$, or $\rho(t) = e^{at+b}$,  for $t > - |c|$. 
\end{theo}
\vskip 2mm
\noindent
{\bf Remark.} 
If $\rho(t) = e^{-t}$ and if $\varphi = e^{-\psi}$ has centroid at $0$, i.e., $\int_{\mathbb{R}^n} x e^{-\psi} dx =0$,  then the inequality of the above theorem simplifies to 
\begin{eqnarray}\label{BLSA}
\int_{\mathbb{R}^n} \rho(\psi(x)) dx  \   \int_{\mathbb{R}^n} \rho(\mathcal{L}_z\psi(x)) dx  &=& \left(\int_{\mathbb{R}^n} e^{-\psi(x))} dx\right)   \  \left( \int_{\mathbb{R}^n} e^{-\psi^*(x))}  dx\right) \nonumber \\
&\leq & \left( \int_{\mathbb{R}^n}  e^{-\frac{\|x\|_2^2}{2}}dx\right)^2.
\end{eqnarray}
\vskip 2mm
 Barthe, B\"or\"oczky and Fradelizi  \cite{BaBoFr} established the following  stability theorem for the  functional 
Blaschke Santal\'o inequality.

\begin{theo} \label{BaBoFr}\cite{BaBoFr} 
Let $\rho : \mathbb{R} \rightarrow \mathbb{R}_+$ be a log-concave and decreasing function with $\int_{\mathbb{R}_+} \rho < \infty$. 
Let $\psi: \mathbb{R}^n \rightarrow \mathbb{R}$ be a convex, measurable function. Assume that for some $\varepsilon \in (0, \varepsilon_0)$ and all
$z \in \mathbb{R}^n$ the following inequality holds
$$
\int_{\mathbb{R}^n} \rho(\psi(x)) dx  \   \int_{\mathbb{R}^n} \rho(\mathcal{L}_z\psi(x)) dx > (1-\varepsilon) \left( \int_{\mathbb{R}^n}  \rho\left(\frac{\|x\|_2^2}{2}\right) dx\right)^2.
$$
Then there exists some $z \in \mathbb{R}^n$, $c \in \mathbb{R}$ and a positive definite $n\times n$ matrix $A$ such that 
$$
\int_{R(\varepsilon) B^n_2} \left|  \frac{\|x\|_2^2}{2} +c - \psi (Ax +z) \right| dx < \eta \varepsilon ^{\frac{1}{129 n^2}}, 
$$
where $\lim_{\varepsilon \rightarrow 0} R(\varepsilon) = \infty$ and $\varepsilon_0, \eta, R(\varepsilon)$ depend on $n$ and $\rho$.
\end{theo}
\vskip 3mm

\subsection {Stability for  Divergence inequalities}

A function $\varphi: \mathbb{R}^n \rightarrow [0, \infty)$ is log concave, if it is of the form $\varphi(x) = e^{-\psi(x)}$, where $\psi: \mathbb{R}^n \rightarrow \mathbb{R}$ is a convex function. Recall that we say that $\varphi=e^{-\psi}$ has centroid at $0$,  respectively  the Santal\'o point, at $0$ if, 
$$
\int x  \varphi (x) dx =\int x e^{-\psi(x)} dx =0,  \  \  \text{respectively} \  \  \int x e^{-\psi^*(x)} dx =0.
$$
The following  entropy inequality for log concave functions was established  in \cite{CaglarWerner2014}, Corollary 13.
\vskip 2mm
\begin{theo} \label{thm00} \cite{CaglarWerner2014}
 Let $ \varphi:\R^{n}\rightarrow [0, \infty)$ be 
 a log-concave
function that  has centroid or Santal\'o point  at $0$.
 Let $f: (0, \infty) \rightarrow \mathbb{R}$ be a convex, decreasing  function. Then 
\begin{eqnarray}\label{thm00,1} 
\int_{\Supp(\varphi)} \varphi  \  f \left( e^{\langle \frac{\nabla \varphi}{\varphi}, x\rangle} \varphi^{-2} \left(  \det \left(  \nabla^2 \left(- \log \varphi\right)\right)\right)\right)
\geq   f \left(   \frac{(2 \pi)^n}{\left(\int \varphi dx \right)^2}   \right) \  \left( \int_{\Supp(\varphi)} \varphi  dx \right).
 \end{eqnarray}
If  $f$ is a concave, increasing function, the inequality is reversed.
 \par
\noindent
Equality holds in both cases if and only if  $\varphi(x)= c  e^{-\langle A x, x \rangle}$, where $c$ is a positive constant and 
$A$ is an $n \times n$  positive definite  matrix.
\end{theo}
Theorem \ref{thm00} was proved under the assumptions that  the convex or concave  functions $f$ and the  log concave  functions 
$ \varphi$ have enough smoothness and integrability properties so that the expressions 
considered in the above statement make sense. Thus, in this section, we will make the same assumptions on $f$ and $\varphi$, 
  i.e., we will  assume that $\varphi^\circ \in L^1(\supp (\varphi), dx)$, the Lebesgue integrable functions on the support of $\varphi$,  that
\begin{equation}\label{assume1}
\varphi \in C^2(\Supp(\varphi))  \cap L^1(\R^n, dx), 
\end{equation}
where  $C^2(\Supp(\varphi))$ denotes the twice continuously differentiable functions on their support, and that 
\begin{equation}\label{assume2}
 \varphi f 
\left(
\frac{e^{\frac{\langle \nabla \varphi, x\rangle}{\varphi}}
}
{\varphi^{2}} 
\mbox{det} \left(  \nabla^2 \left(-\log \varphi \right) 
\right) \right) \in   L^1(\supp (\varphi), dx).
\end{equation}
\par
Recall that  $\varphi(x) = e^{-\psi(x)}$ and put $d\mu=e^{-\psi} dx$. Then  the left hand side of  inequality (\ref{thm00,1}) can  be written as 
$$
\int_{\mathbb{R}^n}    f \left( e^{2 \psi - \langle \nabla \psi,x\rangle}   \   \det \left( \nabla^2 \psi \right) 
\right) d\mu.
$$
It was shown in \cite{CaglarWerner2014} that the left hand side of the inequality (\ref{thm00,1}) is   the natural definition of $f$-divergence $D_f(\varphi)$ for  a log concave
function $\varphi$,  so that  (\ref{thm00,1}) can be rewritten as 
\begin{equation}\label{thm00,2}
D_f(\varphi) \geq  \   f \left(   \frac{(2 \pi)^n}{\left(\int \varphi dx \right)^2}   \right) \
 \  \left( \int_{\Supp(\varphi)} \varphi  dx \right).
 \end{equation}
In information theory, probability theory and statistics, an
$f$-divergence is a function that measures the difference between two (probability)
distributions. We refer to e.g., \cite{CaglarWerner2014} for details and references about $f$-divergence.
\vskip 2mm
\begin{theo} \label{f-stable}
 Let $f: (0, \infty) \rightarrow \mathbb{R}$ be a concave, strictly increasing function.  Let $ \psi:\R^{n}\rightarrow {\R}$ be a
convex function such that   $ e^{- \psi} \in C^2(\R^n)$ and such that $\int_{\mathbb{R}^n} x e^{-\psi(x)}dx=0$ or $\int_{\mathbb{R}^n} x e^{-\psi^*(x)} dx = 0$. Suppose that for some $\varepsilon \in (0, \varepsilon_0)$, 
\begin{eqnarray*}
&& \hskip -7mm \int_{\mathbb{R}^n}    f \left( e^{2 \psi - \langle \nabla \psi,x\rangle}   \   \det \left( \nabla^2 \psi \right) 
\right) d\mu  > \\
&& \hskip 5mm f \left( \frac{ \left(2 \pi\right)^n  }{\left(\int_{\mathbb{R}^n} d\mu  \right)^2 } \right) \left(\int_{\mathbb{R}^n} d \mu  \right)  - \varepsilon  f' \left(  \frac{ \left(2 \pi\right)^n  }{\left(\int_{\mathbb{R}^n} d\mu  \right)^2 } \right)
    \left( \int_{\mathbb{R}^n} d \mu \right)^{-1}.
 \end{eqnarray*}
Then  there exists $c>0$ and a positive definite matrix $A$ such that 
$$
\int_{R(\varepsilon) B^n_2} \left|  \frac{\|x\|_2^2}{2} +c - \psi (Ax ) \right| dx < \eta \varepsilon ^{\frac{1}{129 n^2}}, 
$$
where $\lim_{\varepsilon \rightarrow 0} R(\varepsilon) = \infty$ and $\varepsilon_0, \eta, R(\varepsilon)$ depend on $n$.
\par
The analogue stability result holds, if $f$ is convex and strictly decreasing.
\end{theo}
\vskip 2mm
\begin{proof} We treat the case when $f$ is concave and strictly increasing. The case when $f$ is convex and strictly decreasing is  done similarly.
We set 
$d\nu =\frac{ e^{-\psi} dx}{\int e^{-\psi} dx} =\frac{ \mu}{\int d \mu} $. Then $\nu$ is a probability measure and  
by Jensen's inequality and a change of variable, 
\begin{eqnarray*}
&& \hskip -5mm \left( \int d \mu\right) \int_{\mathbb {R}^n}  f\left( e^{ (2 \psi (x)  - \langle \nabla \psi, x\rangle) } \left(  \det \left(  \nabla^2 \psi(x) \right)\right) 
\right) d\nu \  \  \leq \\
&& \hskip -5mm  \left( \int d \mu\right)  \, f \left( \int_{\mathbb {R}^n}  e^{ (2 \psi (x)  - \langle \nabla \psi, x\rangle) } \left(  \det \left(  \nabla^2 \psi(x) \right)\right) 
 d\nu \right)
 =   f \left( \frac{1}{\int d \mu} \int_{\mathbb {R}^n}  e^{-  \psi ^*(x) } dx\right) \  \left( \int d \mu\right) .
\end{eqnarray*}
Thus, by the assumption of the theorem, we get
\begin{eqnarray*}
 f \left(\frac{1}{ \int d \mu } \int_{\mathbb {R}^n}  e^{ - \psi ^*(x) } dx\right) \ \left( \int d \mu\right)  &> &  \left( \int d \mu\right) \  f \left( \frac{\left(2 \pi \right)^{n}}{\left( \int d \mu\right) ^2}\right)  - \frac{\varepsilon}{ \int d \mu } \  f^\prime\left(\frac{\left(2 \pi \right)^{n}}{\left( \int d \mu\right) ^2}\right) \\&\geq& \left( \int d \mu\right) \ f \left( \frac{\left(2 \pi \right)^{n} - \varepsilon}{\left( \int d \mu\right) ^2}\right).
\end{eqnarray*}
The last inequality holds as by Taylor's theorem and the assumptions on $f$ (i.e., $f''\le 0$), for $\varepsilon$ small enough, there is a real number $\tau$ such that 
\begin{eqnarray*}
 f \left(  \frac{\left(2 \pi \right)^{n} - \varepsilon}{\left( \int d \mu\right) ^2} \right) &=& f \left(  \frac{\left(2 \pi \right)^{n}}{\left( \int d \mu\right) ^2} \right)  - \frac{\varepsilon}{\left( \int d \mu\right) ^2} \ f'\left( \frac{\left(2 \pi \right)^{n} }{\left( \int d \mu\right)^2}\right)    + \frac{\varepsilon^2}{2\left( \int d \mu\right) ^4} \ f''\left( \tau \right)  \\ &\leq& f \left(  \frac{\left(2 \pi \right)^{n}}{\left( \int d \mu\right) ^2} \right)  - \frac{\varepsilon}{\left( \int d \mu\right) ^2} \ f'\left( \frac{\left(2 \pi \right)^{n}}{\left( \int d \mu\right) ^2}\right).
\end{eqnarray*}
Therefore we arrive at
$$f \left( \frac{1}{ \int d \mu } \int_{\mathbb {R}^n}  e^{-  \psi ^*(x) } dx\right)>  f \left( \frac{\left(2 \pi \right)^{n} - \varepsilon}{\left( \int d \mu\right)^2}\right).$$
Since $f$ is strictly increasing we conclude that 
$$\frac{1}{ \int d \mu\ } \int_{\mathbb {R}^n}  e^{-  \psi ^*(x) } dx> \frac{\left(2 \pi \right)^{n} - \varepsilon}{\left( \int d \mu\right)^2},$$
which is equivalent to  
$$
 \left(\int_{\mathbb {R}^n}  e^{ -\psi (x) } dx\right) \left(\int_{\mathbb {R}^n}  e^{  -\psi ^*(x) } dx\right) > \left(2 \pi \right)^{n} - \varepsilon.
$$
From that we get, 
$$
 \left(\int_{\mathbb {R}^n}  e^{ -\psi (x) } dx\right) \left(\int_{\mathbb {R}^n}  e^{  -\psi ^*(x) } dx\right) >(1- \varepsilon) \left(2 \pi \right)^{n}.
$$
As  $\mu$ has its centroid at $0$, we  have by (\ref{BLSA}) that
\begin{eqnarray*}
 \inf_{z \in \mathbb{R}^n}  \left( \int_{\mathbb{R}^n} e^{-\psi(x)} dx \right) \left( \int_{\mathbb{R}^n}e^{-\mathcal{L}_z \psi(y)} \, dy \right) 
 = \left( \int_{\mathbb{R}^n} e^{-\psi} dx \right) \left( \int_{\mathbb{R}^n}e^{-\psi^*(y)} \, dy \right)
\end{eqnarray*}
and  the theorem follows from the result  by Barthe, B\"or\"oczky and Fradelizi \cite{BaBoFr}, Theorem \ref{BaBoFr},  with $\rho(t)=e^{-t}$.
\end{proof}

\subsection{Stability for the reverse log Sobolev inequality}

We now  prove  a stability result for the reverse log Sobolev inequality.  This inequality was first proved  by 
Artstein-Avidan, Klartag, Sch\"utt and Werner \cite{ArtKlarSchuWer} under strong smoothness assumptions. Those were subsequently removed in \cite{CFGLSW} and there,
also equality characterization was achieved. 
\par
We  first recall the the reverse log Sobolev inequality. 
Let $\gamma_n$ be the standard Gaussian measure on $\R^n$. For a  log-concave probability measure $\mu$ on $\R^n$ with density $e^{-\psi}$, i.e.,  $\psi = - \log ( d \mu / dx )$,  let 
\[
 S ( \mu ) =  \int_{\R^n} \psi \, d\mu  
\]
be the Shannon entropy of $\mu$.

\begin{theo}\label{rev-log-sob}\cite{ArtKlarSchuWer, CFGLSW}
Let $\mu$ be a log-concave probability measure on $\mathbb{R}^n$ with density $e^{-\psi}$ with respect to the Lebesgue measure. 
Then 
\begin{equation}\label{Main1} 
\int_{\mathbb{R}^n} \log \left( \det ( \nabla^2 \psi) \right) \, d\mu 
\leq 2\  \left(  S (  \gamma_n ) - S ( \mu ) \right) . 
\end{equation}
Equality holds if and only if $\mu$ is Gaussian (with arbitrary mean and 
 positive definite covariance matrix).
\end{theo}
 Inequality  (\ref{Main1}) is  a reverse log Sobolev inequality as it can be shown that  the log Sobolev inequality is  equivalent to
 \begin{equation*}\label{sobolev}
2\  \Big(  S (  \gamma_n ) - S ( \mu ) \Big)
\leq   n  \log \left(  \frac {  \int_{\R^n} \Delta \psi \, d\mu } n \right) ,
\end{equation*}
where $\Delta$ is the Laplacian.  We refer to e.g., \cite{ArtKlarSchuWer, CFGLSW} for the details.
\par 
Note that inequality (\ref{Main1}) follows from inequality (\ref{thm00,1}) with $f(t)= \log t$. However, because of the  assumptions  on $\varphi$ in Theorem \ref{f-stable},
the result would only hold under those assumptions and not in the full generality stated in Theorem  \ref{rev-log-sob}.
Similarly, a stability result for Theorem  \ref{rev-log-sob}  follows from Theorem \ref{f-stable} with $f(t)= \log t$. But again,  because of the assumptions of Theorem \ref{f-stable}, the result  would only hold  for those $\psi $ such  that  $e^{-\psi}$ is  in $C^2(\mathbb{R}^n)$  and has centroid at $0$.
We can prove a stability result  for Theorem  \ref{rev-log-sob}  without these assumptions. 
The proof  is similar to the one of 
Theorem \ref{f-stable}.
We include it for completeness.  But first we need to recall various  items.
\par
For a convex function $\psi: \R^n \to \R \cup \{+\infty\}$, we define $D_\psi$ to be the convex domain of $\psi$, $D_\psi=\{x \in \R^n, \psi(x) < +\infty\}$. We always consider convex functions $\psi$ such that $\text{int }\left(D_\psi\right)  \ne \emptyset$. 
In the general case, when $\psi$ is neither smooth nor strictly convex, the gradient of $\psi$, denoted  by $\nabla \psi$, exists almost everywhere by Rademacher's theorem (e.g., \cite{Rademacher}),   and
a theorem of Alexandrov \cite{Alex}, Busemann and Feller \cite{BusFel}, guarantees the existence of its Hessian 
$\nabla^2 \psi$ almost everywhere in $\text{int }\left(D_\psi\right)$. We let $X_\psi$  be the set of points of $\text{int }\left(D_\psi\right) $ at which its Hessian $\nabla^2\psi$ in the sense of Alexandrov, Busemann and Feller exists and is invertible.
Then, by definition of the Legendre transform, for a convex function $\psi: \R^n \to \R \cup \{+\infty\}$ 
we  have 
\[
\psi ( x) + \psi^* (y ) \geq \langle x,y \rangle
\] 
for every $x,y \in \R^n$, and with equality if and only if
 $x \in D_\psi$  and $y=  \nabla \psi (x) $, i.e., 
\begin{equation}\label{legendreequality-gen}
 \psi^* ( \nabla \psi (x) ) = \langle x ,\nabla \psi (x) \rangle - \psi (x)  ,\quad \rm{a.e.\ in}\  D_\psi. 
\end{equation}
\vskip 2mm
\begin{theo}
Let $\psi: \mathbb{R}^n \rightarrow \mathbb{R}\cup\{+\infty\}$ be a convex function  and 
let $\mu$ be a log-concave probability measure on $\mathbb{R}^n$ with density $e^{-\psi}$ with respect to Lebesgue measure. 
Suppose that for some $\varepsilon \in (0, \varepsilon_0)$, 
$$
\int_{\mathbb{R}^n}  \log \bigl( \det ( \nabla^2 \psi) \bigr) \, d\mu 
> 2\  \Big(  S (  \gamma_n ) - S ( \mu ) \Big) - \varepsilon.
$$
Then there exists $c>0$ and a positive definite matrix $A$ such that 
$$
\int_{R(\varepsilon) B^n_2} \left|  \frac{\|x\|_2^2}{2} +c - \psi (Ax ) \right| dx < \eta \varepsilon ^{\frac{1}{129 n^2}}, 
$$
where $\lim_{\varepsilon \rightarrow 0} R(\varepsilon) = \infty$ and $\varepsilon_0, \eta, R(\varepsilon)$ depend on $n$.
\end{theo}
\begin{proof}
Both terms of the inequality are invariant under translations of
the measure $\mu$, so we can assume
that $\mu$ has its centroid at $0$.  
\par
Put $\varepsilon = \log \beta >0$. Since $S(\gamma_n)=\frac{n}{2}\log(2\pi e)$,  the inequality of the theorem turns into
\begin{equation*}
\int_{D_\psi} \log \bigl(  \beta \  \det ( \nabla^2 \psi) \bigr) \, d\mu  + 2 \ \int_{D_\psi} \psi \, d\mu >  \log (2 \pi  e)^n,
\end{equation*}
which, in turn is equivalent to
\begin{equation}\label{stab0}
\int_{D_\psi} \log \bigl(  \beta \  \det ( \nabla^2 \psi) \bigr) \, d\mu +  \int_{D_\psi}  \log \left( e^{2\psi} \right)  \, d\mu - n  >  \log (2 \pi )^n.
\end{equation}
We now  use the divergence theorem and get
\[
\int_{D_\psi} \langle x,\nabla \psi (x) \rangle \, d\mu  
= \int_{\text{int} (D_\psi)}  \mathrm{div} ( x ) \, d\mu  - \int_{\partial D_\psi} \langle x, N_{D_\psi} (x)\rangle e^{-\psi(x)}d\sigma_{D_\psi},
\]
where $N_{D_\psi}(x)$ is an exterior normal to the convex set $D_\psi$ at the point $x$ and $\sigma_{D_\psi}$ is the surface area measure on $\partial D_\psi$. Since $D_\psi$ is convex, the centroid $0$ of $\mu$ is in $D_\psi$. Thus $\langle x, N_{D_\psi}(x)\rangle\ge 0$ for every $x\in\partial D_\psi$ and $ \mathrm{div}(x)=n$ hence 
\[
-n \leq - \int_{D_\psi} \langle x,\nabla \psi (x)\rangle  \, d\mu= \int _{D_\psi} \log \left( e^{- \langle x, \nabla \psi(x) \rangle}\right) d \mu
\]
Thus we get from inequality (\ref{stab0}), 
\begin{equation*}
\int_{D_\psi} \log \bigl(  \beta \  \det ( \nabla^2 \psi)  \  e^{2 \psi(x) - \langle x, \nabla \psi(x) \rangle}\bigr ) \, d\mu  >  \log (2 \pi )^n.
\end{equation*}
With Jensen's inequality, and as $d \mu = e^{-\psi} dx$, 
\begin{equation}\label{stab1}
\beta \int_{D_\psi} \det ( \nabla^2 \psi)  \  e^{ \psi(x) - \langle x, \nabla \psi(x) \rangle} dx   >   (2 \pi )^n.
\end{equation}
By (\ref{legendreequality-gen}), 
$$
\int_{D_\psi} \det ( \nabla^2 \psi)  \  e^{ \psi(x) - \langle x, \nabla \psi(x) \rangle} dx = \int_{D_\psi} \det ( \nabla^2 \psi)  \  e^{ -\psi^*(\nabla \psi (x))} dx. 
$$
The change of variable $y= \nabla \psi (x) $  gives
\begin{equation}\label{variablechangemccann}
\int_{D_\psi}  e^{-\psi^*(\nabla \psi (x))} \det ( \nabla^2 \psi (x) )\, dx= \int_{D_{\psi^*}} e^{-\psi^*(y)} \, dy,
\end{equation}
and 
inequality (\ref{stab1}) becomes 
\begin{equation*}
 \int_{D_{\psi^*}} e^{-\psi^*(y)} \, dy > \frac{1}{\beta} (2 \pi )^n.
 \end{equation*}
As $\int_{D_\psi} e^{-\psi} dx =1$ and $\beta^{-1}=e^{-\varepsilon}\ge 1-\varepsilon$,  we therefore get that 
 \begin{eqnarray*}
 \left( \int_{\mathbb{R}^n} e^{-\psi} dx \right) \left( \int_{\mathbb{R}^n}e^{-\psi^*(y)} \, dy \right)  &\geq & \left( \int_{D_\psi} e^{-\psi} dx \right) \left( \int_{D_{\psi^*}}e^{-\psi^*(y)} \, dy \right)\\
 & >& (1-\varepsilon) (2 \pi )^n.
 \end{eqnarray*}
As  $\mu$ has its centroid  at $0$, we  have by (\ref{BLSA}) that 
\begin{eqnarray*}
 \inf_{z \in \mathbb{R}^n}  \left( \int_{\mathbb{R}^n} e^{-\psi(x)} dx \right) \left( \int_{\mathbb{R}^n}e^{-\mathcal{L}_z \psi(y)} \, dy \right)  
= \left( \int_{\mathbb{R}^n} e^{-\psi} dx \right) \left( \int_{\mathbb{R}^n}e^{-\psi^*(y)} \, dy \right).
\end{eqnarray*}
The theorem now follows from Theorem  \ref{BaBoFr}, the stability result for the functional Blaschke Santal\'o inequality,  due to  Barthe, B\"or\"oczky and Fradelizi \cite{BaBoFr}.

\end{proof}

\subsection{Stability for the $L_\lam$-affine isoperimetric inequality for log concave functions} \label{Llambda}
The following divergence inequalities were proved in \cite{CaglarWerner2014}. In fact, inequalities (\ref{lambda1}), (\ref{lambda2}) and consequently  (\ref{thm00,1})  are special cases of a more general
divergence inequality proved in \cite{CaglarWerner2014}.

For $ 0 \leq \lambda \leq 1$, it says
\begin{equation}\label{lambda1}
\int_{} \left( e^{2 \psi - \langle \nabla \psi,x\rangle}   \   \det \left( \nabla^2 \psi \right) 
\right) ^\lambda d\mu  \leq \  \left(  \frac{\int_{\mathbb{R}^n}  e^{-\psi^*}dx}{ \int_{\mathbb{R}^n} d\mu}   \right)^\lambda
 \  \left( \int_{\mathbb{R}^n}  d\mu \right)
 \end{equation}
and for $\lambda \notin [0,1]$,
\begin{equation}\label{lambda2}
\int_{} \left( e^{2 \psi - \langle \nabla \psi,x\rangle}   \   \det \left( \nabla^2 \psi \right) 
\right) ^\lambda d\mu  \geq \  \left(  \frac{\int_{\mathbb{R}^n}  e^{-\psi^*}dx}{ \int_{\mathbb{R}^n} d\mu}   \right)^\lambda
 \  \left( \int_{\mathbb{R}^n}  d\mu \right).
 \end{equation}
 The left hand sides of the above inequalities are  the $L_\lambda$-affine surface areas $as_\lambda (\psi)$. For 
 a general  log concave function $\varphi = e^{-\psi}$ (and not just 
 a log concave function in $C^2(\mathbb{R}^n)$) they were introduced in  \cite{CFGLSW}, 
 \begin{equation}\label{affine}
 as_\lambda(\psi) = \int_{X_\psi}   e^{(2\lambda - 1)\psi(x) - \lambda  \langle \nabla \psi, x\rangle } \left(  \det \left(  \nabla^2 \psi(x) \right)\right)
^\lambda dx.
\end{equation}
Since $\det \left(  \nabla^2 \psi(x) \right) =0$ outside $X_\psi$, the integral may be taken on $D_\psi$ for $\lambda >0$. 
In particular,   
$$
as_0(\psi)=   \int_{X_\psi}   e^{-\psi(x)} dx \   \   \text{ and }   \   \   as_1(\psi) =  \int_{X_{\psi^*}}   e^{-\psi^*(x)} dx.
$$
Assume  now that $\int x e^{-\psi(x)}dx=0$ or $\int x e^{-\psi^*(x)} dx = 0$.  Then we can apply the functional Blaschke Santal\'o inequality (\ref{BLSA}) and get
from (\ref{lambda1}) that for $\lambda \in [0,1]$, 
\begin{equation*}\label{lambda1,1}
as_\lambda (\psi)  \leq \  \left( 2 \pi \right)^{n \lambda} \left(  \int_{\mathbb{R}^n} e^{-\psi(x)} dx \right)^{1- 2 \lambda}.
 \end{equation*}
 Similarly, for $\lambda \leq 0$, we get from (\ref{lambda2})
 \begin{equation*}\label{lambda2,1}
as_\lambda (\psi)  \geq \  \left( 2 \pi \right)^{n \lambda} \left(  \int_{\mathbb{R}^n} e^{-\psi(x)} dx \right)^{1- 2 \lambda},
 \end{equation*}
 provided that $\varphi \in C^2(\mathbb{R}^n)$, which is the assumption on $\varphi$ in inequality (\ref{thm00,1}). 
However, these  inequalities  hold without such a strong smoothness assumption.  This, together with characterization of equality,  was proved in  \cite{CFGLSW}.
\par
\begin{theo}\label{asa-lam}\cite{CFGLSW}
Let  $\psi: \mathbb{R}^n \rightarrow \mathbb{R} \cup \{\infty\}$ be a convex function.  For $\lambda \in [0,1]$,
\begin{equation} \label{lambda1,1}
as_\lambda (\psi)  \leq \  \left( 2 \pi \right)^{n \lambda} \left(  \int_{X_\psi}  e^{-\psi(x)} dx \right)^{1- 2 \lambda} 
\end{equation}
and for $\lambda \leq 0$,
\begin{equation} \label{lambda2,1}
as_\lambda (\psi)  \geq \  \left( 2 \pi \right)^{n \lambda} \left(  \int_{X_\psi}  e^{-\psi(x)} dx \right)^{1- 2 \lambda}.
\end{equation}
For $\lambda=0$ equality holds trivially in these inequalities. 
Moreover, for  $0 < \lambda \leq 1$, or $\lambda < 0$,  equality holds in above inequalities if and only if  $\psi(x) =  \frac{1}{2 } \langle Ax, x \rangle  +c$, where $A$ is a positive definite $n\times n$ matrix and $c$ is a constant.
\end{theo}
\vskip 3mm
A stability result for these inequalities is again an immediate  consequence of Theorem \ref{f-stable}. But again, we would then get the stability result for 
log concave  functions $\varphi \in C^2(\mathbb{R}^n)$ only, so we include the proof for general functions. 

\vskip 2mm
\begin{prop}\label{stab-affinefunctions}
Let  $\psi: \mathbb{R}^n \rightarrow \mathbb{R} \cup \{+\infty\}$ be a convex function  such that $\int x e^{-\psi(x)}dx=0$ or $\int x e^{-\psi^*(x)} dx = 0$. 
\par
\noindent
(i) Let $0 < \lambda \leq1$ and suppose that for some $\varepsilon \in (0, \varepsilon_0)$, 
$$as_\lambda (\psi) > (1-\varepsilon) ^\lambda \left( 2 \pi \right)^{n \lambda} \left(  \int_{X_\psi} e^{-\psi(x)} dx \right)^{1- 2 \lambda}.$$
\par
\noindent
(ii) Let $\lambda <0$ and suppose that for some $\varepsilon \in (0, \varepsilon_0)$, 
$$as_\lambda (\psi) < (1-\varepsilon) ^\lambda \left( 2 \pi \right)^{n \lambda} \left(  \int_{X_\psi}  e^{-\psi(x)} dx \right)^{1- 2 \lambda}.$$
Then, in both cases (i) and (ii),  there exists $c>0$ and a positive definite matrix $A$ such that 
$$
\int_{R(\varepsilon) B^n_2} \left|  \frac{\|x\|_2^2}{2} +c - \psi (Ax ) \right| dx < \eta \varepsilon ^{\frac{1}{129 n^2}}, 
$$
where $\lim_{\varepsilon \rightarrow 0} R(\varepsilon) = \infty$ and $\varepsilon_0, \eta, R(\varepsilon)$ depend on $n$.
\end{prop}
\vskip 2mm
\begin{proof} 
(i)  The case $\lambda=1$ is the stability case for the functional Blaschke Santal\'o inequality of  Theorem \ref{BaBoFr}. Therefore  we can assume that $0 < \lambda <1$. We put $d \mu = e^{-\psi} dx$.  By H\"older's inequality with $p=1/\lambda$ and $q=1/(1-\lambda)$,
\begin{eqnarray*}
as_\lambda (\psi) &=& \int_{X_\psi}    e^{\lambda (2 \psi (x)  - \langle \nabla \psi, x\rangle) } \left(  \det \left(  \nabla^2 \psi(x) \right)\right)
^\lambda d\mu \\
& \leq&  \left( \int_{X_\psi}   e^{2 \psi (x)  - \langle \nabla \psi, x\rangle }  \det \left(  \nabla^2 \psi(x) \right)
 d\mu\right)^\lambda \left( \int_{X_\psi}  d \mu \right) ^{1- \lambda} \\
&=&  \left( \int_{D_\psi}   e^{ \psi (x)  -  \langle \nabla \psi, x\rangle }  \det \left(  \nabla^2 \psi(x) \right)
 dx\right)^\lambda \left( \int_{X_\psi} e^{-\psi(x)}d x \right) ^{1- \lambda}\\
& \leq&
\left( \int_{\mathbb {R}^n} e^{-\psi^*(x)}d x \right) ^{ \lambda} \left( \int_{X_\psi}  e^{-\psi(x)}d x \right) ^{1- \lambda}, 
\end{eqnarray*}
where, in the last equality, we have used (\ref{legendreequality-gen}) and (\ref{variablechangemccann}).
Therefore, by the  assumption (i)  of the proposition 
\begin{eqnarray*}
\left( \int_{\mathbb {R}^n} e^{-\psi^*(x)}d x \right) ^{ \lambda} \left( \int_{X_\psi}  e^{-\psi(x)}d x \right) ^{1- \lambda} > (1-\varepsilon)^\lambda \left( 2 \pi \right)^{n \lambda} \left(  \int_{X_\psi}  e^{-\psi(x)} dx \right)^{1- 2 \lambda},
\end{eqnarray*}
which is equivalent to 
\begin{eqnarray*}
\left( \int_{\mathbb {R}^n} e^{-\psi^*(x)}d x \right)  \left( \int_{\mathbb {R}^n} e^{-\psi(x)}d x \right) >   \left( \int_{\mathbb {R}^n} e^{-\psi^*(x)}d x \right)  \left( \int_{X_\psi} e^{-\psi(x)}d x \right) > (1-\varepsilon)  \left( 2 \pi \right)^{n },
\end{eqnarray*}
and the result is again a consequence of Theorem \ref{BaBoFr}  by Barthe, B\"or\"oczky and Fradelizi \cite{BaBoFr}.
\par
Similarly, in the case (ii) the proposition follows by applying the reverse H\"older inequality.
\end{proof}
The following Blaschke Santal\'o  type inequality follows directly from inequality (\ref{lambda1,1}). It was also proved, together with its equality characterization in \cite{CFGLSW}.  
\begin{cor} \cite{CFGLSW}
Let $ \lambda \in [0, \frac{1}{2}]$ and let $\psi: \mathbb{R}^n \rightarrow \mathbb{R} \cup \{+\infty\}$ be a convex function such that $\int x e^{-\psi(x)}dx=0$ or $\int x e^{-\psi^*(x)} dx = 0$. Then
$$ 
  as_\lam(\psi) \ as_\lam ( (\psi^* ) )   \le (2\pi)^{n}  .
$$
Equality holds if and only if there exists $a\in\R$ and a positive definite matrix $A$ such that $\psi(x)=\frac{1}{2} \langle Ax,x\rangle+a$, for every $x\in\R^n$.
\end{cor}
\par
We have the following stability result as a direct consequence of
Theorem \ref{BaBoFr}.
\par
\begin{prop}\label{stab-affinefunctions2}
Let  $\psi: \mathbb{R}^n \rightarrow \mathbb{R} \cup \{+\infty\}$ be a convex function such that $\int x e^{-\psi(x)}dx=0$ or $\int x e^{-\psi^*(x)} dx = 0$. 
Let $0 \leq \lambda \leq \frac{1}{2}$ and suppose that for some $\varepsilon \in (0, \varepsilon_0)$, 
 $$ as_\lam(\psi) \ as_\lam ( (\psi^* ) )   \geq (1- \epsilon) (2\pi)^{n} .
 $$
\par
\noindent
Then, there exists $c>0$ and a positive definite matrix $A$ such that 
$$
\int_{R(\varepsilon) B^n_2} \left|  \frac{\|x\|_2^2}{2} +c - \psi (Ax ) \right| dx < \eta \varepsilon ^{\frac{1}{129 n^2}}, 
$$
where $\lim_{\varepsilon \rightarrow 0} R(\varepsilon) = \infty$ and $\varepsilon_0, \eta, R(\varepsilon)$ depend on $n$.
\end{prop}

\vskip 5mm 
\noindent 
Elisabeth M. Werner\\
{\small Department of Mathematics \ \ \ \ \ \ \ \ \ \ \ \ \ \ \ \ \ \ \ Universit\'{e} de Lille 1}\\
{\small Case Western Reserve University \ \ \ \ \ \ \ \ \ \ \ \ \ UFR de Math\'{e}matique }\\
{\small Cleveland, Ohio 44106, U. S. A. \ \ \ \ \ \ \ \ \ \ \ \ \ \ \ 59655 Villeneuve d'Ascq, France}\\
{\small \tt elisabeth.werner@case.edu}\\ \\

\noindent 
Umut Caglar\\
{\small Department of Mathematics and Statistics\\
{\small Florida International University\\
{\small Miami, FL 33199, U. S. A. \\
{\small \tt ucaglar@fiu.edu}\\ \\

\end{document}
\bye